\documentclass[12pt, reqno]{amsart}
\usepackage[utf8]{inputenc}
\usepackage{amsmath}
\usepackage{amsfonts}
\usepackage{amssymb}
\usepackage{amsthm}
\usepackage{bm}
\usepackage{bbm}
\usepackage{xcolor}
\usepackage{mathtools}
\usepackage{url}
\usepackage{hyperref}
\theoremstyle{plain}
\newtheorem{theorem}{Theorem}[section]
\newtheorem{lemma}[theorem]{Lemma}

\theoremstyle{definition}

\theoremstyle{remark}
\newtheorem{remark}[theorem]{Remark}

\numberwithin{equation}{section}

\title{Utilizing Smoothing Techniques to Bound $|\zeta(1+it)|$}

\author{Andrew Christensen}
\email{andrewchristensenj@gmail.com}
\author{Kyle Pratt}
\email{kyle.pratt@mathematics.byu.edu}
\address{Department of Mathematics, Brigham Young University, Provo, UT 84602, USA}


\allowdisplaybreaks

\begin{document}
\date{}

\begin{abstract}
We demonstrate an improved explicit upper bound of $|\zeta(1+it)|$ for $3 \leq t \leq 10^9$ using smoothing techniques. Our method sharpens previous bounds relying on the Riemann--Siegel formula and the triangle inequality. In particular, we prove that for $t\geq 3$, 
\begin{align*}
    |\zeta(1+it)| \leq \frac{1}{2}\log t + 1.57 
\end{align*}
and for $t \geq 10^8$,
\[
|\zeta(1+it)|\leq \frac{1}{3}\log t + 2\log \log t -1.16 .
\]
\end{abstract}

\maketitle

\section{Introduction}

Studying the behavior of the Riemann zeta function on the one-line has been of particular interest for a number of years, largely due to its utility in expanding the zero-free regions of $\zeta(s)$ \cite{FORD_2002}. It is often useful to have quality explicit upper bounds on $|\zeta(1+it)|$, particularly when probing zero-free regions of the zeta function by computational means. Several such bounds have been proved, including in \cite{Patel2020} and \cite{Hiary2025}. A vast majority of these bounds rely heavily on the Riemann--Siegel formula or similar approximate functional equations of the zeta function where the value of $\zeta(1+it)$ is approximated by
\begin{equation*}
\zeta(1+it) \approx \sum_{n=1}^{N(t)}\frac 1{n^{1+it}}
\end{equation*}
for some  $N(t)$. 

For large $t$, previous researchers have applied exponential sum techniques to the Riemann--Siegel formula to achieve quality explicit upper bounds of $|\zeta(1+it)|$. In \cite{Patel2020}, Patel proved
\begin{equation}\label{eq: Patel bound}
|\zeta(1+it)|\leq \min \left( \log t, \frac 12 \log t + 1.93, \frac15\log t + 44.02  \right)
\end{equation}
for $t\geq 3$. The bounds $\log t$ and $\frac 12 \log t + 1.93$ are essentially trivial bounds arising from applying the triangle inequality to the Riemann--Siegel formula. The bound $\frac15\log t + 44.02 $ arises via exponential sum estimates and is better than the previous bounds only when $t$ is sufficiently large.
 
Hiary, Leong, and Yang \cite{Hiary2025} refined the work of Patel and showed
\begin{equation}\label{eq: Hiary bound}
|\zeta(1+it)|\leq 1.731 \frac{\log t}{\log\log t}
\end{equation}
for $t\geq 3$. Hiary, Leong, and Yang use further exponential sum arguments to improve \eqref{eq: Patel bound} for large $t$, but for $t \leq e^{16}\approx 8.88 \cdot 10^6$ they use Patel's bound $|\zeta(1+it)| \leq \frac 12 \log t + 1.93$.

Working with smoothed sums can sometimes yield improved quantitative or qualitative results. One noteworthy example of this philosophy is Helfgott's proof of the ternary Goldbach conjecture (see \cite{Helfgott}), in which smoothed sums appear throughout the argument. We show that utilizing smoothing techniques can improve the ``trivial'' bounds of $|\zeta(1+it)|$ for small $t$. In particular, we obtain improved upper bounds for $t \leq 10^9$.
We essentially work with a smooth representation
\begin{equation}\label{eq: smoothing form}
\zeta(1+it) \approx \sum_{n=1}^\infty{\frac{1}{n^{1+it}}\exp{\left(-\frac{n}{X}\right)}}
\end{equation}
for suitable $X$. The following is our main result.
\begin{theorem}\label{thm: main}
    If $t\geq 3$, then
    \begin{equation}
        |\zeta(1+it)| \leq \frac{1}{2}\log t + 1.57.
    \end{equation}
    Furthermore, if $t \geq 10^8$, then
    \begin{equation}
        |\zeta(1+it)| \leq \frac{1}{3}\log t + 2\log \log t -1.16.
    \end{equation}
\end{theorem}

Theorem \ref{thm: main} improves upon \eqref{eq: Patel bound} and \eqref{eq: Hiary bound} for $t\leq 10^9$.  While both  bounds rely heavily on smoothing, the bound
\[
|\zeta(1+it)|\leq  \frac{1}{3}\log t + 2\log \log t -1.16
\]
also depends in part on exponential sum techniques. However, the bound
\[
|\zeta(1+it)|\leq\frac{1}{2}\log t + 1.57
\]
relies only on a trivial bound for $|\zeta(1/2+it)|$. This bound, without using exponential sums, is an improvement of \eqref{eq: Patel bound} and \eqref{eq: Hiary bound} for  $t\leq 4 \cdot 10^8$.

In this paper, we focus on smoothings of the form seen in \eqref{eq: smoothing form}, but it would be interesting to consider the effects of different smoothing functions. It is possible that using exponential sums and smoothing techniques in tandem yields even better results. It would also be worthwhile to see if using smoothings could sharpen bounds on $|\zeta^{-1}(1+it)|$ or $|\frac{\zeta'}{\zeta}(1+it)|$ in ranges of interest.

In Section \ref{sec: section 2} we prove a smoothed representation of $\zeta(1+it)$ of the form seen in \eqref{eq: smoothing form} and show that that smoothed sum can be bounded by $\log X$ plus some small error. In Section \ref{sec: section 3}, we twice bound our main error term, once optimizing for smaller $t$ and once for larger $t$. Finally, in Section \ref{sec: proof of theorem} we prove Theorem \ref{thm: main}.

Some of our work relies on computer calculation. All the code used in this paper is available at our ``Utilizing Smoothing Techniques to Bound $\zeta(1+it)$'' GitHub repository \cite{ChristensenPratt2026}.

\section{A Smoothed Representation of $\zeta(1+it)$} \label{sec: section 2}

We begin by proving a smoothed approximation of $\zeta(1+it)$ of the form given in \eqref{eq: smoothing form}.

\begin{lemma}\label{lem: main bound}
    For all $t \in \mathbb R\setminus\{0\}$ and $X>0$,
    \begin{align} \label{eq: main bound}
        |\zeta(1+it)|\leq&\sum_{n=1}^\infty{\frac{1}{n}\exp{\left(-\frac{n}{X}\right)}} +\left|\Gamma(-it)\right| \\ &+\frac{X^{-1/2}}{2\pi }\int_{- \infty}^{\infty}\left|\Gamma(-1/2+iu)\right| \cdot \left|  \zeta(1/2+i(t+u))\right|du. \notag
    \end{align}
\end{lemma}

\begin{proof}
    Since the inverse Mellin transform of $\Gamma(s)$ for $\operatorname{Re}(s)>0$ is $e^{-y}$ \cite[2.5.1]{Oberhettinger_1974}, we can write
    \begin{equation*}
    \frac 1{2\pi i} \int_{2-i\infty}^{2+i\infty}y^s \Gamma(s) \ ds  = \exp\left(-\frac 1y\right).
    \end{equation*}
    Taking $y=X/n$ and summing over $n$, we can write
    \begin{align*}
        \sum_{n=1}^\infty \frac{1}{n^{1+it}}\exp\left(-\frac nX\right) & = \sum_{n=1}^\infty \frac{1}{n^{1+it}}\frac 1{2\pi i} \int_{2-i\infty}^{2+i\infty}\left(\frac X n\right)^s \Gamma(s) \ ds \\
        &= \frac{1}{2\pi i}  \int_{2-i\infty}^{2+i\infty} X^s\Gamma(s) \sum_{n=1}^\infty \frac{1}{n^{1+it+s}}ds \\
        &=\frac{1}{2\pi i}  \int_{2-i\infty}^{2+i\infty} X^s\Gamma(s) \zeta(1+it+s) \ ds,
    \end{align*}
    where interchanging the order of summation and integration is justified by absolute convergence.
    
    We shift the line of integration to $\text{Re}(s) = -\frac{1}{2}$, picking up contributions from poles at $s=0$ and $s=-it$. Hence
    \begin{align*}
         \sum_{n=1}^\infty& \frac{1}{n^{1+it}}\exp\left(-\frac nX\right)-\frac{1}{2\pi i}  \int_{-1/2-i\infty}^{-1/2+i\infty} X^s\Gamma(s) \zeta(1+it+s) \ ds =\zeta(1+it) + X^{-it}\Gamma(-it),
    \end{align*}
    and rearranging gives
    \begin{equation*}
        \zeta(1+it) = \sum_{n=1}^\infty \frac{1}{n^{1+it}}\exp\left(-\frac nX\right)-X^{-it}\Gamma(-it)-\frac{1}{2\pi i}  \int_{-1/2-i\infty}^{-1/2+i\infty} X^s\Gamma(s) \zeta(1+it+s) \ ds.
    \end{equation*}
    Applying the triangle inequality yields
    \begin{equation*}
        |\zeta(1+it)|\leq\sum_{n=1}^\infty{\frac{1}{n}\exp{\left(-\frac{n}{X}\right)}} +\left|\Gamma(-it)\right| + \frac{1}{2\pi }\int_{-1/2- i \infty}^{-1/2+ i \infty}\left|{X^s \Gamma(s)} \zeta(1+it+s)\right|ds.
    \end{equation*}
    We perform a change of variables in the integral to obtain
    \begin{align*}
    \frac{1}{2\pi }\int_{-1/2- i \infty}^{-1/2+ i \infty}\left|{X^s \Gamma(s)} \zeta(1+it+s)\right|ds &= \frac{1}{2\pi }\int_{- \infty}^{\infty}\left|{X^{-1/2+iu} \Gamma(-1/2+iu)}  \zeta(1/2+i(t+u))\right|du \\
    &= \frac{X^{-1/2}}{2\pi }\int_{- \infty}^{\infty}\left|\Gamma(-1/2+iu)\right| \cdot \left|  \zeta(1/2+i(t+u))\right|du,
    \end{align*}
    which gives the desired bound.
\end{proof}

We now proceed by bounding the terms on the right-hand side of Lemma \ref{lem: main bound} individually. We begin with the summation. 

\begin{lemma}\label{lem: sum term bound}
    If $X>0$, $\xi \coloneqq e-2.5$, and
    \[
    c_0 \coloneqq \frac{12 \xi -1}{24}, \ \ c_1 \coloneqq \frac{60 \xi^2-1}{120},
    \]
    then
    \begin{equation} \label{eq:Bound on the sum term}
         \sum_{n=1}^\infty \frac{1}{n} \exp\left( -\frac{n}{X} \right) \leq \log X +\frac{1}{2X}-\frac{1}{24X^2} + \frac{c_0}{X^3} + \frac{c_1}{X^4}.
    \end{equation}
\end{lemma}

\begin{proof}
    Since $X>0$, it follows that $|\exp(-1/X)|<1$.  Therefore,
    \begin{align}\label{eq:Taylor expansion of sum}
        \sum_{n=1}^\infty{\frac{1}{n}\exp{\left(-\frac{n}{X}\right)}} &=-\log \left( 1-e^{-1/X} \right) = \log X - \log\left(X\left(e^{1/X}-1\right)\right)+\frac{1}{X}.
    \end{align}
    We write $X\left( e^{1/X}-1 \right)$ as a Taylor series to obtain
    \begin{align*}
    -\log\left( X\left( e^{1/X}-1 \right) \right) &= -\log\left( 1+ \frac{1}{2X} + \frac{1}{6X^2} + \ldots \right).
    \end{align*}
    For ease of notation, we define 
    \begin{align*}
    B&\coloneqq\sum_{k=1}^\infty \frac{1}{(k+1)! X^k} =\frac{1}{2X} + \frac{1}{6X^2} + \frac{1}{24X^3} + \ldots,\\
    C&\coloneqq\sum_{k=2}^\infty \frac{1}{(k+1)! X^k} = \frac{1}{6X^2} + \frac{1}{24X^3}  + \frac{1}{120X^4}  + \ldots,
    \end{align*}
    so that
    \[
    -\log\left( X\left( e^{1/X}-1 \right) \right) = -\log\left( 1 + B  \right) = -\log\left( 1+ \frac{1}{2X} +C \right).
    \]
    Furthermore, note that we can write
    \[
    B= \frac{1}{2X} + C, \ \ \ B^2 = \frac{1}{4 X^2} + \frac{C}{X} + C^2.
    \]
    Expanding as a Taylor series gives
    \begin{align*}
    -\log\left( X\left( e^{1/X}-1 \right) \right)&=-\log\left( 1 + B  \right)= -B+\frac{B^2}{2} - \frac{B^3}{3} + \ldots\\
    &\leq -B+\frac{B^2}{2}=-B +\frac{1}{8 X^2} + \frac{C}{2X} + \frac{C^2}{2}.
    \end{align*}
    We can bound $C$ by
    \[
    C=\sum_{k=2}^\infty \frac{1}{(k+1)! X^k} \leq \frac{1}{X^2}\sum_{k=3}^\infty \frac{1}{k!} = \frac{e-2.5}{X^2} =\frac{\xi}{X^2}.
    \]
    Applying this bound, we find
    \begin{align*}
         -\log\left( X\left( e^{1/X}-1 \right) \right) &\leq -\frac{1}{2X} - \frac{1}{6X^2} - \frac{1}{24X^3} - \frac{1}{120X^4} -\ldots + \frac{1}{8 X^2} + \frac{\xi}{2X^3} + \frac{\xi^2}{2 X^4} \\
        &\leq -\frac{1}{2X}-\frac{1}{24X^2}  + \frac{12\xi - 1}{24 X^3} + \frac{60 \xi^2-1}{120 X^4}.
    \end{align*}
    Using this,  \eqref{eq:Taylor expansion of sum} becomes
    \begin{align*}
    \sum_{n=1}^\infty \frac{1}{n} \exp\left( -\frac{n}{X} \right) &\leq \log X -\frac{1}{2X}-\frac{1}{24X^2} + \frac{c_0}{X^3} + \frac{c_1}{X^4} +\frac{1}{X}  \\
    &=\log X +\frac{1}{2X}-\frac{1}{24X^2} + \frac{c_0}{X^3} + \frac{c_1}{X^4}. \qedhere
    \end{align*}
\end{proof}

\section{Bounding the Integral} \label{sec: section 3}

To aid in bounding the integral in \eqref{eq: main bound}, we make use of two previously proven bounds on the Riemann zeta function. 
\begin{lemma}\label{lem:bounds on 1/2}
    If $t \geq 200$, then
    \begin{equation}\label{eq:1/4bound}
        |\zeta(1/2+it)| \leq \frac{4 t^{1/4}}{(2\pi)^{1/4}}-2.08.
    \end{equation}
    Furthermore, for $t\geq 3$, we have
    \begin{equation}\label{eq:1/6bound}
        |\zeta(1/2+it)| \leq 0.618 t^{1/6} \log t.
    \end{equation}
\end{lemma}

\begin{proof}
    The first bound \eqref{eq:1/4bound}, proved in \cite{Hiary2015}, is an improved version of the Riemann--Siegel--Lehman bound. The bound \eqref{eq:1/6bound} is proved  in \cite{HiaryPatelYang2022}, and is based on exponential sum techniques.
\end{proof}

While \eqref{eq:1/4bound} is essentially a trivial bound, the bound \eqref{eq:1/6bound} relies on exponential sums. The trivial bound is better for $t\lessapprox 5 \cdot 10^7$ while \eqref{eq:1/6bound} is better for larger $t$. 

As \eqref{eq:1/4bound} and \eqref{eq:1/6bound} are both of the form $|\zeta(1/2+it)|\leq c_2 t^a(\log t)^b-c_3$, we bound the integral in \eqref{eq: main bound} using an arbitrary bound on $|\zeta(1/2+it)|$ of this form.

For the duration of this paper, we fix $A_0 \coloneqq 200$.

\begin{lemma}\label{lem:genral integral}
    If there exist  $a,b,c_2,c_3 \in \mathbb R_{\geq0}$ with $a\leq 1$ such that
    \begin{equation}\label{eq:general zeta bound}
        |\zeta(1/2+iw)|\leq c_2 w^a(\log w)^b-c_3
    \end{equation}
    for all $w\geq 210$, then for all $t\geq 210$ we have
    \begin{align}
    \int_{- \infty}^{\infty}\left|\Gamma(-1/2+iu)\right| \cdot &\left|  \zeta(1/2+i(t+u))\right|du \leq  c_2\big(t^aI_1+I_2\big)+ I_3 + 2.27\cdot10^{-5},
    \end{align}
    where 
    \begin{align*}
    I_1 &\coloneqq \int_{A_0}^{\infty}\left| \Gamma(-1/2+iu)\right|  \log^b(2u) du,  \ \ \ \ \ \ \ \   I_2 \coloneqq \int_{A_0}^{\infty}\left| \Gamma(-1/2+iu)\right|    u^a \log^b(2u) du. \\
     I_3&\coloneqq\int_{A_0-t}^{\infty}\left| \Gamma(-1/2+iu)\right| \cdot \left| \zeta(1/2+i(t+u))\right|du.
    \end{align*}
\end{lemma}

\begin{remark}
    Since our choice of $A_0$ is sufficiently large, the variable $u$ in the integrals $I_1$ and $I_2$ is far enough from the origin that the exponential decay of the gamma function causes their value to be small. As such, the effect of these integrals on our error is minimal. Most of our error will arise from $I_3$, which is large since its range contains the origin. In an attempt to maximize our savings we bound $I_3$ twice, using both bounds of $|\zeta(1/2+it)|$ found in Lemma \ref{lem:bounds on 1/2}. The bound in \eqref{eq:1/6bound} will give better error for large $t$, while \eqref{eq:1/4bound} will be more effective for small values of $t$. 
\end{remark}

\begin{proof}[Proof of Lemma \ref{lem:genral integral}]
    Due to the conjugate symmetry of the zeta function, it follows that the bound $|\zeta(1/2-iw)|\leq c_2 w^a\log^b w-c_3$ holds for all $w\geq 210$. As such, we define
    \begin{align*}
    I_4 &\coloneqq \int_{-\infty}^{-A_0-t}\left| \Gamma(-1/2+iu)\right| \cdot \left| \zeta(1/2+i(t+u))\right|du, \\
    I_5 &\coloneqq\int_{-A_0-t}^{A_0-t}\left| \Gamma(-1/2+iu)\right| \cdot \left| \zeta(1/2+i(t+u))\right|du,
    \end{align*}
    and split the integral
    \begin{align}\label{eq:splitting the integral}
         \int_{- \infty}^{\infty}\left|\Gamma(-1/2+iu)\right| \cdot \left|  \zeta(1/2+i(t+u))\right|du=& I_3+I_4+I_5.
    \end{align}
    We bound $I_4$ and $I_5$ separately. 

    First, consider $I_4$. Using our bound \eqref{eq:general zeta bound}, we have
    \[
    I_4 \leq I_6 - I_7
    \]
    where
    \begin{align*}
        I_6 &\coloneqq c_2\int_{-\infty}^{-A_0-t}\left| \Gamma(-1/2+iu)\right| \cdot  |t+u|^a \log^b|t+u| du \\
        I_7 &\coloneqq c_3\int_{-\infty}^{-A_0-t}\left| \Gamma(-1/2+iu)\right| du.
    \end{align*}
    Consider $I_6$. Then $t \leq |u|$, so we can bound
    \[
    \log|t+u|\leq \log(t+|u|) \leq \log(2|u|).
    \]
    Furthermore, since $a \leq 1$, we have $|t+u|^a \leq t^a + |u|^a$. Thus we have
    \begin{align*}
    I_6 &\leq c_2\int_{-\infty}^{-A_0-t}\left| \Gamma(-1/2+iu)\right|  \left(t^a + |u|^a\right) \log^b(2|u|) du \\
    &= c_2t^a\int_{-\infty}^{-A_0-t}\left| \Gamma(-1/2+iu)\right|  \log^b(2|u|) du + c_2\int_{-\infty}^{-A_0-t}\left| \Gamma(-1/2+iu)\right|    |u|^a \log^b(2|u|) du.
    \end{align*}
    To eliminate the $t$-dependence in the bounds of integration, we can bound $-A_0-t \leq -A_0$ and write
    \begin{align*}
    I_6 &\leq c_2t^a\int_{-\infty}^{-A_0}\left| \Gamma(-1/2+iu)\right|  \log^b(2|u|) du + c_2\int_{-\infty}^{-A_0}\left| \Gamma(-1/2+iu)\right|    |u|^a \log^b(2|u|) du.
    \end{align*}
    However, due to the conjugate symmetry of the gamma function, these integrals are exactly $I_1$ and $I_2$, so we have 
    \[
    I_6 \leq c_2\big(t^aI_1+I_2\big).
    \]
    For $I_7$, since its integrand and $c_3$ are nonnegative, we can simply bound $I_7 \geq 0$. Thus we have
    \begin{equation}\label{eq:lower integral bound}
    I_4 \leq I_6-I_7 \leq  c_2\big(t^aI_1+I_2\big).
    \end{equation}

    Next consider $I_5$. Letting $\omega=t+u$, it becomes
    \[
    I_5=\int_{-A_0}^{A_0}|\Gamma(-1/2+i(\omega-t))|\cdot|\zeta(1/2+i\omega)|d\omega.
    \]
    Noting the relation \cite[5.4.4]{NIST:DLMF}
    \begin{equation}\label{eq: Gamma bound}
    |\Gamma(1/2 + iy)|^2 = \frac\pi{\cosh(\pi y)},
    \end{equation}
    we use the functional equation of the gamma function to find
    \[
    |\Gamma(-1/2+i(\omega-t))| = \Bigg[ \frac{\pi}{\cosh(\pi(\omega-t))} \frac{1}{\frac{1}{4}+(\omega-t)^2} \Bigg]^{1/2}.
    \]
    Therefore we have
    \begin{align*}
    I_5 &= \int_{-A_0}^{A_0} \Bigg[ \frac{\pi}{\cosh(\pi(\omega-t))} \frac{1}{\frac{1}{4}+(\omega-t)^2} \Bigg]^{1/2}  |\zeta(1/2+i\omega)|d\omega \\
    &\leq \Bigg[ \frac{\pi}{\cosh(\pi(t-A_0))} \frac{1}{\frac{1}{4}+(t-A_0)^2} \Bigg]^{1/2} \int_{-A_0}^{A_0}|\zeta(1/2+i\omega)|d\omega.
    \end{align*}
    Bounding this trivially (using $t \geq 210$) and evaluating the integral numerically in \cite{ChristensenPratt2026} gives 
    \begin{equation}\label{eq:bound on middle integral}
        I_5 \leq2.27\cdot10^{-5}.
    \end{equation}

    Combining our results from \eqref{eq:splitting the integral}, \eqref{eq:lower integral bound}, and \eqref{eq:bound on middle integral} we arrive at the desired expression:
    \begin{align*}
    \int_{- \infty}^{\infty}\left|\Gamma(-1/2+iu)\right| \cdot \left|  \zeta(1/2+i(t+u))\right|du &\leq c_2\big(t^aI_1+I_2\big)+ I_3 + 2.27\cdot10^{-5}.\qedhere
    \end{align*}
\end{proof}

The integrals $I_1$ and $I_2$ can both be evaluated numerically, with the tail bounded analytically for various choices of $a$ and $b$. These bounds are somewhat tedious and uninteresting so we do not give all the details. However, we will demonstrate a bound of the tail of $I_1$ with $b=0$ and note that we can bound the others quite similarly. 
\begin{lemma}\label{lem: bounding tail}
    If $b=0$, then
    \[
    I_1=\int_{A_0}^{\infty} |\Gamma(-1/2+iu)|du\leq 10^{-133}.
    \]
\end{lemma}
\begin{proof}
    We can split $I_1$ into two regions:
    \[
    I_1=\int_{A_0}^{\infty} |\Gamma(-1/2+iu)|du = \int_{A_0}^{B_0} |\Gamma(-1/2+iu)|du+\int_{B_0}^{\infty} |\Gamma(-1/2+iu)|du.
    \]
    Taking $B_0 = 1000$, we bound the first integral using complex ball arithmetic in \cite{ChristensenPratt2026} and find
    \[
    \int_{A_0}^{B_0} |\Gamma(-1/2+iu)|du \leq 10^{-134}.
    \]
    For the tail, we can use the functional equation of the gamma function along with \eqref{eq: Gamma bound} to write
    \[
    I_1'\coloneqq\int_{B_0}^{\infty} |\Gamma(-1/2  + iu)|du = \int_{B_0}^{\infty}  \frac{|\Gamma(1/2+iu)|}{|-1/2 + iu|}du = \int_{B_0}^{\infty}  \frac{\pi^{1/2}}{|-1/2 + iu| \sqrt{\cosh(\pi u)}}du.
    \]
    We then quickly bound the resulting integral in as follows:
    \begin{align*}
        I_1' &\leq 2\sqrt{2\pi}\int_{B_0}^{\infty}  \frac{du}{\left(e^{\pi u} + e^{-\pi u}  \right)^{1/2}} \leq 2\sqrt{2\pi} \int_{B_0}^{\infty}  e^{-\pi u/2}du = \frac {4\sqrt 2}{\sqrt\pi}e^{-B_0\pi/2} \leq 10^{-681}.
    \end{align*}
    With this, we find
    \[
    I_1 \leq 10^{-134}+10^{-681} \leq 10^{-133}.\qedhere
    \]
\end{proof}

    We can similarly bound  both $I_1$ and $I_2$ for various values of $a$ and $b$. However, the integral $I_3$ still has some $t$ dependence. Since $I_3$ provides the largest contribution to our eventual error term, we treat it carefully.  In the following two lemmas, we provide two different bounds on $I_3$, the first using the bound in \eqref{eq:1/4bound} and the second using the bound in \eqref{eq:1/6bound}.

\begin{lemma}\label{lem: 1/2 upper integral}
    If $t\geq 210$, then
    \begin{equation}\label{eq: 1/2 upper integral bound}
        I_3 \leq 4.0315 c_2 t^{1/4} + 0.4618c_2t^{-3/4}-4.0314 c_3,
    \end{equation}
    where $c_2 = 4/(2\pi)^{1/4}$ and $c_3 = 2.08$.
\end{lemma}

\begin{proof}

    Since $t+u \geq t+(A_0-t) = A_0 = 200$, Lemma \ref{lem:bounds on 1/2} yields
    \[
    \left| \zeta(1/2+i(t+u))\right| \leq c_2(t+u)^{1/4}-c_3.
    \]
    It follows that
    \[
    I_3 \leq c_2\int_{A_0-t}^\infty |\Gamma(-1/2+iu)|(t+u)^{1/4}du - c_3 \int_{A_0-t}^\infty |\Gamma(-1/2+iu)|du.
    \]
    We define 
    \begin{align*}
        I_8 &\coloneqq c_2\int_{A_0-t}^\infty |\Gamma(-1/2+iu)|(t+u)^{1/4}du, \\
        I_9 &\coloneqq  c_3 \int_{A_0-t}^\infty |\Gamma(-1/2+iu)|du,
    \end{align*}
    so that $I_3 \leq I_8- I_9$. 
    
    We first examine $I_8$. Pulling out a factor of $t^{1/4}$ gives
    \[
    I_8 = c_2t^{1/4}\int_{A_0-t}^\infty |\Gamma(-1/2+iu)|\left(1+\frac ut\right)^{1/4}du.
    \]
    Since $u/t\geq -1$, we see by concavity that $(1+u/t)^{1/4}\leq 1+u/4t$. Inserting this bound and expanding the integral gives
    \begin{align*}
        I_8 \leq& c_2t^{1/4}\int_{A_0-t}^\infty |\Gamma(-1/2+iu)|du+\frac{c_2}{4t^{3/4}}\int_{A_0-t}^\infty |\Gamma(-1/2+iu)| u  du.
    \end{align*}
    We obtain an upper bound on the integrals, and simultaneously remove their $t$-dependence, by extending the integrals to the entire real line. Hence
    \begin{align*}
        I_8 \leq& c_2 t^{1/4}\int_{-\infty}^\infty |\Gamma(-1/2+iu)|du + \frac{c_2}{4 t^{3/4}}\int_{-\infty}^\infty |\Gamma(-1/2+iu)||u|du.
    \end{align*}
    Rigorously integrating these numerically using SageMath in \cite{ChristensenPratt2026} and bounding the tail  (similarly to the proof of Lemma \ref{lem: bounding tail}),  we find
    \[
    I_8\leq 4.0315 c_2 t^{1/4} + 0.4618c_2t^{-3/4}.
    \]
    Similarly we get a lower bound $I_9 \geq 4.0314 c_3$, and we arrive at the desired conclusion.
\end{proof}

\begin{lemma}\label{lem: 1/3 upper integral}
    If $t\geq 210$ and $c_2 = 0.618$, then 
    \begin{align}\label{eq: 1/3 upper integral bound}
        I_3\leq c_2 \Bigg[ 4.0315 t^{1/6} \log t + 1.8469 \  t^{-5/6} + 3.3032 \log t + 1.7493 \  t^{-1}   \Bigg].
    \end{align}
\end{lemma}

\begin{proof}
Since $u \geq A_0-t$, Lemma \ref{lem:bounds on 1/2} gives
\[
|\zeta(1/2+i(t+u))| \leq c_2 (t+u)^{1/6} \log (t+u).
\]
Therefore, we have
\[
I_3 \leq c_2\int_{A_0-t}^{\infty}|\Gamma(-1/2 + iu)| (t+u)^{1/6}\log(t+u)du.
\]
Since $u \geq A_0-t > -t$, we have $\frac{u}{t} > -1$, and therefore
\[
\log(t+u) =  \log t + \log\left(1+\frac ut\right) \leq \log t + \frac{|u|}{t}.
\]
We also have the upper bound
\[
(t+u)^{1/6} \leq t^{1/6} + |u|^{1/6}.
\]
Inserting both of these bounds gives
\[
I_3 \leq c_2 \int_{A_0-t}^{\infty}|\Gamma(-1/2  +iu)| (t^{1/6} + |u|^{1/6})\left( \log t + \frac {|u|}t\right) du .
\]
Extending the range of integration to the whole real line and expanding gives
\begin{align*}
    I_3 \leq&c_2 \Bigg[ t^{1/6} \log t \int_{-\infty}^\infty |\Gamma(-1/2 +iu)| du  + t^{-5/6} \int_{-\infty}^\infty |\Gamma(-1/2 +iu)|  |u| du  \\
    &  + \log t \int_{-\infty}^\infty |\Gamma(-1/2 +iu)||u|^{1/6} du + t^{-1}\int_{-\infty}^\infty |\Gamma(-1/2 +iu)| |u|^{7/6} du    \Bigg].
\end{align*}
The first two integrals we can bound the tails and rigorously integrate using SageMath in the usual way. For the later two integrals, since $|u|^{1/6}$ and $|u|^{7/6}$ are not holomorphic at $u=0$, we must also analytically bound the integral in a very small neighborhood of the origin before performing numerical integration.  Performing the computations in \cite{ChristensenPratt2026}, we arrive at 
\[
I_3\leq c_2 \Bigg[ 4.0315 t^{1/6} \log t + 1.8469 \  t^{-5/6} + 3.3032 \log t + 1.7493 \  t^{-1}   \Bigg]. \qedhere
\]
\end{proof}

\section{Proof of Theorem 1.1} \label{sec: proof of theorem}

\begin{proof}[Proof of Theorem \ref{thm: main}]
    With bounds on $I_3$, we can now directly bound the integral term in \eqref{eq: main bound}. We apply Lemma \ref{lem:genral integral} to \eqref{eq:1/4bound} and \eqref{eq:1/6bound} to arrive at two distinct bounds of $|\zeta(1+it)|$. We begin with the bound in \eqref{eq:1/4bound}. That is, setting 
    \[
    a=\frac 14, \ \ b = 0,  \ \ c_2 = \frac 4{(2\pi)^{1/4}}, \ \ c_3= 2.08,
    \]
    Lemma \ref{lem:genral integral} gives
    \begin{align*}
         I &\coloneqq \int_{- \infty}^{\infty}\left|\Gamma(-1/2+iu)\right| \cdot \left|  \zeta(1/2+i(t+u))\right|du \leq c_2\big(t^{1/4}I_1+I_2\big)+ I_3 + 2.27\cdot10^{-5}.
    \end{align*}
    With these parameters, we numerically bound $I_1$ and $I_2$ in \cite{ChristensenPratt2026} and find
    \[
    I_1 \leq 10^{-134}, \ \ \ I_2 \leq 10^{-131}.
    \]
    Furthermore, Lemma \ref{lem: 1/2 upper integral} gives
    \[
    I_3  \leq 4.0315 c_2 t^{1/4} + 0.4618c_2t^{-3/4}-4.0314 c_3.
    \]
    Hence defining
    \[
    c_4\coloneqq 10.18549 , \ \ \ c_5\coloneqq1.16673, \ \ \ c_6\coloneqq 8.38527,
    \]
    we have 
    \begin{align*}
        I  &\leq 4.031501\ c_2t^{1/4} + 0.4618c_2t^{-3/4}+ c_2 \cdot 10^{-131}   -4.0314 c_3 + 2.27 \cdot 10^{-5} \leq  c_4 t^{1/4} + c_5 t^{-3/4} - c_6.
    \end{align*}
    This along with Lemmas \ref{lem: main bound} and \ref{lem: sum term bound} allows us to write
    \begin{align*}
        |\zeta(1+it)|&\leq \log X +\frac{1}{2X}-\frac{1}{24X^2} + \frac{c_0}{X^3} + \frac{c_1}{X^4} + |\Gamma(-it)| + \frac{X^{-1/2}}{2\pi } \left(c_4 t^{1/4} + c_5 t^{-3/4} - c_6\right) \\
        &\coloneqq \log X  + \mathcal R_1(X).
    \end{align*}
    The two largest terms in this bound are $\log X$ and $\frac{c_4}{2\pi}\frac{t^{1/4}}{X^{1/2}}$. As such, we minimize the sum of these terms by choosing
    \[
    X \coloneqq \frac{c_4^2 t^{1/2}}{(4\pi)^2}.
    \]
    Thus we have
    \begin{align*}
        |\zeta(1+it)| & \leq \frac 12 \log t  +\log \left( \frac{c_4^2}{(4\pi)^2}\right) + \mathcal R_1\left(\frac{c_4^2 t^{1/2}}{(4\pi)^2} \right),
    \end{align*}
    say. Since
    \[
    \log \left( \frac{c_4^2}{(4\pi)^2}\right) + \mathcal R_1\left(\frac{c_4^2 t^{1/2}}{(4\pi)^2} \right)
    \]
    is increasing on $t\geq 10$, we bound the expression trivially for $210 \leq t \leq 5 \cdot 10^8$ in \cite{ChristensenPratt2026} to find
    \begin{equation}\label{eq: 1/2 log bound}
    |\zeta(1+it)| \leq \frac 12 \log t + 1.57.
    \end{equation}

    Similarly, we can bound the integral term using the bound given in \eqref{eq:1/6bound}. Here we have
    \[
    a= \frac 16, \ \ \ b=1, \ \ \ c_2 = 0.618, \ \ \ c_3=0.
    \]
    With these parameters, we numerically bound $I_1$ and $I_2$ in \cite{ChristensenPratt2026} as follows:
    \[
    I_1\leq 10^{-133}, \ \ \ I_2 \leq 10^{-130}.
    \]
    Furthermore, Lemma \ref{lem: 1/3 upper integral} gives
    \[
    I_3\leq c_2 \Bigg[ 4.0315 t^{1/6} \log t + 1.8469 \  t^{-5/6} + 3.3032 \log t + 1.7493 \  t^{-1}   \Bigg].
    \]
    Thus if we define
    \begin{align*}
    c_7&\coloneqq 2.49147, & c_8 &\coloneqq  2.04138, & c_9&\coloneqq10^{-133}, \\
    c_{10}&\coloneqq1.14139, & c_{11}&\coloneqq 1.08107, & c_{12}&\coloneqq2.271\cdot10^{-5},
    \end{align*}
    we have
    \begin{align*}
        I &\leq  4.0315 c_2t^{1/6} \log t +3.3032c_2 \log t + c_2 10^{-133} t^{1/6} \\
        & \ \ \ \ + 1.8469  \ c_2 t^{-5/6}  + 1.7493 \  c_2t^{-1} + c_2 10^{-130} + 2.27 \cdot 10^{-5} \\
        &\leq  c_7t^{1/6} \log t + c_8 \log t + c_9t^{1/6} + c_{10}t^{-5/6} + c_{11}t^{-1} + c_{12}.
    \end{align*}
    This along with Lemmas \ref{lem: main bound} and \ref{lem: sum term bound} allows us to write 
    \begin{align*}
        |\zeta(1+it)|&\leq \log X +\frac{1}{2X}-\frac{1}{24X^2} + \frac{c_0}{X^3} + \frac{c_1}{X^4} + |\Gamma(-it)| \\
        & \ \ \ \ + \frac{X^{-1/2}}{2\pi } \left(c_7t^{1/6} \log t + c_8 \log t + c_9t^{1/6} + c_{10}t^{-5/6} + c_{11}t^{-1} + c_{12}\right).
    \end{align*}
    The two largest terms in this bound are $\log X$ and $\frac{c_7}{2\pi X^{1/2}}t^{1/6}\log t$. As such, we minimize the sum of these terms by choosing
    \[
    X \coloneqq \frac{c_7^2\  t^{1/3}\log^2t}{(4\pi)^2}.
    \]
    Thus we have
    \[
    |\zeta(1+it)| \leq \frac13\log t + 2\log\log t + \log\left(\frac{c_7^2}{(4\pi)^2}\right) + \mathcal R_2\left(\frac{c_7^2\  t^{1/3}\log^2t}{(4\pi)^2}\right),
    \]
    say. Since
    \[
    \log\left(\frac{c_7^2}{(4\pi)^2}\right) + \mathcal R_2\left(\frac{c_7^2\  t^{1/3}\log^2t}{(4\pi)^2}\right)
    \]
    is decreasing on $t \geq 10$,  for  $t \geq 10^8$ we trivially bound this expression in \cite{ChristensenPratt2026} to arrive at
    \begin{equation}\label{eq: 1/3 log bound}
    |\zeta(1+it)| \leq \frac13 \log t+2\log\log t-1.16.
    \end{equation}
    Noting that the bound in \eqref{eq: 1/2 log bound} is greater than the bound in \eqref{eq: 1/3 log bound} for all $t \geq 2 \cdot 10^8$, it follows that for all $t \geq 210$ 
    \[
    |\zeta(1+it)| \leq \frac 12 \log t + 1.57.
    \]
    Furthermore, we prove computationally in \cite{ChristensenPratt2026} that this bound holds for all $3 \leq t \leq 210$. Hence we have the desired bound.  \qedhere
\end{proof}

\section{Acknowledgments}
Both authors are partially supported by the National Science Foundation (DMS-2418328). The second author is also partially supported by the Simons Foundation (MPS-TSM-00007959).

\bibliographystyle{plain}
\bibliography{references}

\end{document}